\documentclass[12pt]{article}

\usepackage{amsmath}
\usepackage{amsfonts}
\usepackage{amssymb}
\usepackage{amscd}
\usepackage{amsthm}
\usepackage{bbm}
\usepackage{color}      
\usepackage[linktocpage=true,plainpages=false,pdfpagelabels=false]{hyperref}

\input xypic

\usepackage[matrix,arrow,curve]{xy}

\usepackage{color}


\newtheorem{prop}{Proposition}
\newtheorem{nt}{Remark}
\newtheorem{Th}{Theorem}
\newtheorem{lemma}{Lemma}

\newfont{\ssdbl}{msbm8}
\newfont{\sdbl}{msbm9}
\newfont{\dbl}{msbm10 at 12pt}
\newcommand{\eqdef}{\stackrel{\rm def}{=}}

\newcommand{\oo}{{\cal O}}

\newcommand{\Dim}{\mathop {\rm Dim}}

\newcommand{\Lim}{\mathop {\rm lim}}

\newcommand{\Spec}{\mathop {\rm Spec}}

\newcommand{\Frac}{\mathop {\rm Frac}}

\newcommand{\da}{\mathbb{A}}

\newcommand{\dr}{\mathbb{R}}

\newcommand{\Z}{\dz}

\newcommand{\lrto}{\longrightarrow}

\newcommand{\df}{\mathbb{F}}

\def\R{{\mathbb R}}
\def\Z{{\mathbb Z}}
\def\Q{{\mathbb Q}}

\newcommand{\diag}{{\rm diag}}

\newcommand{\mC}{{\mathcal C}}
\newcommand{\mD}{{\mathcal D}}

\begin{document}

\author{
D. V. Osipov
}

\title{Second Chern numbers of vector bundles and higher adeles\footnotetext{This work is supported by the Russian Science Foundation  under  grant 14-50-00005.}}
\date{}

\maketitle

\begin{abstract}
We give a construction of the second Chern number of a vector bundle over a smooth projective surface by means of adelic transition matrices for the vector bundle. The construction does not use an algebraic $K$-theory and depends on the canonical $\Z$-torsor of a locally linearly compact $k$-vector space. Analogs of certain auxiliary results for the case of an arithmetic surface are also discussed.
\end{abstract}

\section{Introduction}

In~\cite{Par} A.N.~Parshin constructed Chern classes of vector bundles on a scheme $Y$ which is finite type over the field $\Q$ using higher adeles. In particular,
Chern classes, which he constructed, were in $H^m (Y, \Omega_Y^m)$. Taking the higher residues when $m = \dim Y$, we obtain the Chern numbers,
see~\cite[\S~4.3]{Par}.  This construction can be carried out when $Y$ is a scheme over any field $k$, but because of the higher residues the values of the Chern numbers of vector bundles
will be in the image of the ring $\Z$ in the field $k$. Thus, if $\mathop{\rm char} k =p >0$, then we will obtain the Chern numbers only modulo $p$.

Much later
there appeared adelic constructions of second Chern classes on certain two-dimensional regular schemes be means of $K_2$-groups. In particular,
R.Ya.~Budylin in~\cite{Bu} constructed the second Chern classes of vector  bundles of rank $2$ on a smooth algebraic surface $Y$ over any perfect field using $K_2$-groups of rational adeles on $Y$. Besides, T.~Chinburg, G.~Pappas and M. J.~Taylor gave in~\cite{CPT} a construction of the second Chern classes of vector bundles of arbitrary rank on a regular two-dimensional scheme $Y$ with  projective morphism of relative dimension $1$ to the spectrum of a Dedekind ring
by means of $K_2$-adeles on $Y$ originated from~\cite{O0}.

In this paper we provide a quite elementary construction of the second Chern numbers of vector bundles on a smooth projective surface $X$ over a perfect field $k$. This construction does not use algebraic $K$-theory, but uses only $\Z$-torsors and central extensions of a group $GL_n(\da_X)$
by the group $\Z$, where $\da_X$ is the adelic ring of $X$, which is also called the Parshin-Belinson adeles of $X$.

More exactly,  any locally linearly compact vector space over a field $k$ gives a canonical $\Z$-torsor of dimension theories. The adelic space $\da_X$ has a filtration given by divisors on $X$ with the quotient spaces being locally linearly compact vector spaces over $k$. The same is also true for $\da_X^n$ for any integer $n \ge 0$. Therefore from the action of the group $GL_n(\da_X)$ on the $k$-vector space $\da_X^n$ we obtain a canonical central extensions
$\widetilde{GL_n(\da_X)}$ and then  $\widehat{GL_n(\da_X)}$
of this group by the group $\Z$. The trivializations of a vector bundle at scheme points of $X$ give transition matrices which are elements of $GL_n(\da_X)$ and satisfy the cocycle condition. Using  canonical splittings of the central extension  $\widehat{GL_n(\da_X)}$ over certain subgroups
of $GL_n(\da_X)$, we obtain lifts of these transition matrices to $\widehat{GL_n(\da_X)}$, where their product is an element over $1 \in GL_n(\da_X)$, i.e. it belongs to the subgroup $\Z$.
This is the second Chern number of the vector bundle, see Theorem~\ref{Chern}.

The advantage of our approach is similarity to the constructions from~\cite{OsiPa}, where an "analytic" proof of the Riemann-Roch theorem for linear bundles on a smooth projective surface $X$ over a finite field was given. One of the main ingredients in this proof was the definition of the intersection index of two divisors on $X$
via the commutator of lifts of certain elements from the group $\da_X^*$ to a central extension which is similar to the central extension
$\widetilde{GL_1(\da_X)}$.  We note that the Noether formula was not obtained in~\cite{OsiPa}. Therefore one of the first expected applications of our construction of the second Chern numbers will be the proof of the Noether formula in the spirit of~\cite{OsiPa}.

The next direction for the applications is the transfer of our constructions to the case of an arithmetic surface such that the fibres over Archimedean points of the base are taken into account. In particular,
 in the case of an arithmetic surface $X$ over $\Spec \Z$ and the adelic ring $\da_X^{\rm ar}$ which
includes an adelic object of the fibre over $\infty$-point of $\Spec \Z$,
 we prove  in this paper in Proposition~\ref{arithm}  splittings of central extensions  $\widetilde{GL_n(\da^{\rm ar}_X)}$ and $\widehat{GL_n(\da^{\rm ar}_X)}$ over certain subgroups of $GL_n(\da^{\rm ar}_X)$.
  These splittings are analogs of splittings considered above for the construction of second Chern number of a vector bundle over  an algebraic surface. The central extensions  $\widetilde{GL_n(\da^{\rm ar}_X)}$ and $\widehat{GL_n(\da^{\rm ar}_X)}$ are central extensions  by the group of positive real numbers $\dr_+^*$ and were also  considered in~\cite{Osi2}.

 The paper is organized as follows. In Section~\ref{adeles} we recall certain facts on the Parshin-Beilinson adeles of an algebraic surface $X$.
 In section~\ref{dimens} we recall the notion of $\Z$-torsor of dimension theories for a locally linearly compact $k$-vector space.
 In section~\ref{centr-ext1} we give a construction of the central extension $\widetilde{GL_n(\da_{\Delta})}$, where $\da_{\Delta}$
 is the adelic ring which depends on a subset  $\Delta$ of all pairs $x \in C$, where $x$ is a point and $C$ is an irreducible curve on $X$.
   In section~\ref{sec_comm} we connect the commutator of lifts of elements from $\da_{X}^*$ to $\widetilde{GL_1(\da_X)}$ with the intersection index of divisors on $X$ by proving a result which was given without proof in~\cite{OsiPa}, see Proposition~\ref{intersect}. In section~\ref{centr-ext2}
   we give a construction of the central extension~$\widehat{GL_n(\da_{\Delta})}$ and prove some properties of this central extension, see Proposition~\ref{ext-restr}. In section~\ref{can-split} we prove canonical splittings of the central extensions  $\widetilde{GL_n(\da_X)}$ and  $\widehat{GL_n(\da_X)}$   over certain subgroups, see Proposition~\ref{split-ext}. In section~\ref{Chern-sect} we give a construction of the second Chern number, see Theorem~\ref{Chern}. In section~\ref{arsur} we prove certain results on splittings of central extensions in the case of an arithmetic surface, see Proposition~\ref{arithm}.

\section{Central extension and intersection index of divisors} \label{sec2}
\subsection{Parshin-Beilinson adeles}  \label{adeles}

Let $X$ be a smooth algebraic surface over a perfect field $k$.
Let $\da_X$ be the Parshin-Beilinson adelic ring of $X$ (see, for example,  a survey in~\cite{Osi}).

Let $x \in C$ be a  pair, where $x$ is a  point on $X$, and $C$ is an irreducible curve on $X$ such that $C$ contains $x$. Let $K_{x,C}= \prod_{i=1}^l K_i$, where an index $i$ corresponds to a formal irreducible  branch $C_i$ of the curve $C$ in the formal neighbourhood of $x$
(i.e. $C \mid_{\Spec \hat{\oo}_{x,X}} = \bigcup\limits_{i=1}^l C_i$, where $\hat{\oo}_{x,X}$ is the completion of the local ring $\oo_{x,X}$ of $x$ on $X$), and $K_i$ is a two-dimensional local field that is the completion of the  fraction field $\Frac \hat{\oo}_{x,X}$ with respect to the discrete valuation given by $C_i$.

We note that
\begin{equation}  \label{adel}
\da_X   \subset \prod_{x \in C} K_{x,C}  \mbox{,}
\end{equation}
where the product is over all pairs $x \in C$ described as above.

Let $\Delta$ be a subset in the set of all pairs $x \in C$ described as above. There are the following subrings of the ring $\prod_{x \in C} K_{x,C}$:
\begin{equation} \label{adel-delta}
\da_{\Delta}= \da_X   \cap \prod_{ \{x \in C\} \in \Delta } K_{x,C} \mbox{,}  \qquad \oo_{\da_{\Delta}}  = \da_X \cap \prod_{ \{x \in C\} \in \Delta } \oo_{K_{x,C}}  \mbox{,}
\end{equation}
where $\oo_{K_{x,C}} = \prod_{i=1}^l  \oo_{K_i}$, and $\oo_{K_i}$ is the discrete valuation ring of the field $K_i$. Clearly, if $\Delta$ is the set of all pairs $x \in C$, then $\da_{\Delta}= \da_X$. Moreover, if $\Delta$ is a single pair $x \in C$, then $\da_{\Delta}=K_{x,C}$.

Let $D = \sum_{i} a_i C_i$ be a divisor on $X$. (Here  $a_i \in \Z$ and $C_i$ is an irreducible curve on $X$ for any $i$). We call
$a_i = \nu_{C_i} (D)$ for any $i$. We define
$$
\oo_{\da_{\Delta}}(D) =\da_X \cap \prod_{ \{x \in C\} \in \Delta } t_C^{-\nu_C(D)} \oo_{K_{x,C}} \mbox{,}
$$
where $t_C = 0$ is an equation of an irreducible curve $C$ on some open subset of $X$. (The definition of $\oo_{\da_{\Delta}}(D)$ does not depend on the choice of $t_C$.)

We note (see~\cite[prop.~2.1.5]{H}) that if $\Delta = \Delta_1 \cup \Delta_2$ and $\Delta_1 \cap \Delta_2 = \emptyset$, then
$$
\da_{\Delta} = \da_{\Delta_1}  \times \da_{\Delta_2} \mbox{,}   \qquad    \oo_{\da_{\Delta}}=   \oo_{\da_{\Delta_1}}  \times \oo_{\da_{\Delta_2}}  \mbox{.}
$$
Hence we obtain  for any integer $n \ge 1$
\begin{equation}  \label{decomp}
GL_n(\da_{\Delta})= GL_n(\da_{\Delta_1})  \times GL_n(\da_{\Delta_2})  \mbox{.}
\end{equation}

\subsection{Dimension theories}  \label{dimens}

Our first goal  is to construct  central extensions $\widetilde{GL_n(\da_{\Delta})}$ and $\widehat{GL_n(\da_{\Delta})}$  of the group $GL_n(\da_{\Delta})$
by the group $\Z$. These central extensions are similar to central extensions $\widetilde{GL_n(\da_{\Delta})}_{\dr_+^*}$ and $\widehat{GL_n(\da_{\Delta})}_{\dr_+^*}$
from~\cite[\S~3]{Osi2}. (More close relation will be given in Section~\ref{arsur} below.) The main tool for this construction is a $\Z$-torsor $\Dim$ of dimension theories on a locally linearly compact $k$-vector space $V$ (or, in other words, on $1$-Tate $k$-vector space~$V$). This $\Z$-torsor was defined by M.~Kapranov in~\cite{K}.

We recall the definition of $\Dim(V)$.  A dimension theory $d$ on $V$ is a map from the set of all open linearly compact $k$-subspaces of $V$ to the group $\Z$ such that
$d(U_2)= d(U_1) + \dim_k(U_2/U_1)$ whenever $U_2 \supset U_1$ are two open linearly compact $k$-subspaces of $V$. (We note that $\dim_k(U_2/ U_1) < \infty$.) The set of all dimension theories on $V$ is denoted by $\Dim(V)$. The group $\Z$ acts on $\Dim(V)$ by adding constant maps. This makes $\Dim(V)$ into a $\Z$-torsor.

We consider an exact sequence of $k$-vector spaces
\begin{equation}  \label{ex_seq}
0 \lrto V_1 \lrto V_2 \lrto V_3  \lrto 0 \mbox{,}
\end{equation}
where $V_i$ ($1 \le i \le 3$) are locally linearly compact $k$-vector spaces and all the maps in the above sequence are continuous. Besides, let $V_1$ be a closed subspace of $V_2$, and topology on $V_3$ coincides with  the quotient topology. In this case, there is a canonical isomorphism
\begin{equation} \label{dim-ext}
\Dim (V_1)  \otimes_{\Z}  \Dim(V_3)  \lrto \Dim(V_2)
\end{equation}
given as $d_1 \otimes d_3 \mapsto d_2$, where $d_2(U)= d_1(U \cap V_1) + d_3 (U/ (U \cap V_1))$ for a linearly compact subspace $U$ of $V_2$.

\subsection{Central extension $\widetilde{GL_n(\da_{\Delta})}$} \label{centr-ext1}

By construction, $$\da_{\Delta} = \mathop{\Lim_{\lrto}}_{D_1} \mathop{\Lim_{\longleftarrow}}_{D_2 \le D_1}  \oo_{\da_{\Delta}}(D_1)/ \oo_{\da_{\Delta}}(D_2)    \mbox{,}$$
and the $k$-vector space $\oo_{\da_{\Delta}}(D_1)/ \oo_{\da_{\Delta}}(D_2)  $ is a locally linearly compact $k$-vector space for any divisors $D_2 \le D_1$ on $X$. Besides, for any divisors $D_1 \ge D_2 \ge D_3$ on $X$  the corresponding exact sequence
$$
0 \lrto \oo_{\da_{\Delta}}(D_2)/ \oo_{\da_{\Delta}}(D_3)  \lrto \oo_{\da_{\Delta}}(D_1)/ \oo_{\da_{\Delta}}(D_3) \lrto \oo_{\da_{\Delta}}(D_1)/ \oo_{\da_{\Delta}}(D_2)  \lrto 0
$$
has the same properties as the exact sequence~\eqref{ex_seq}. This means that $\da_{\Delta}$, and correspondingly $\da_{\Delta}^n$, is a complete $C_2$-vector space over $k$ (or a $2$-Tate vector space over $k$) from~\cite{Osip}. In particular, for any elements $g_1 $ and $g_2$ from $GL_n(\da_{\Delta})$ such that $g_1 \oo_{\da_{\Delta}}^n  \subset g_2 \oo_{\da_{\Delta}}^n$ we have that the $k$-vector space
$g_2 \oo_{\da_{\Delta}}^n  / g_1 \oo_{\da_{\Delta}}^n$ is a locally linearly compact with the  induced and quotient topology  from
 a locally linearly compact $k$-vector space $\oo_{\da_{\Delta}}(D_1)^n / \oo_{\da_{\Delta}}(D_2)^n$ for appropriate divisors $D_1 \ge D_2$ on $X$.    Therefore a $\Z$-torsor
\begin{equation}  \label{dim1}
\Dim(g_1\oo_{\da_{\Delta}}^n \mid g_2\oo_{\da_{\Delta}}^n )  \eqdef \Dim( g_2 \oo_{\da_{\Delta}}^n  / g_1 \oo_{\da_{\Delta}}^n )
\end{equation}
is well-defined. We define also
\begin{equation}   \label{dim2}
\Dim(g_2\oo_{\da_{\Delta}}^n \mid g_1\oo_{\da_{\Delta}}^n )  \eqdef \Dim( g_2 \oo_{\da_{\Delta}}^n  / g_1 \oo_{\da_{\Delta}}^n )^{\vee}  \mbox{,}
\end{equation}
where the sign $\vee$ means the dual $\Z$-torsor.
Now for any elements $g_1$ and $g_2$ from $GL_n(\da_{\Delta})$ a $\Z$-torsor $\Dim(g_1\oo_{\da_{\Delta}}^n \mid g_2\oo_{\da_{\Delta}}^n)$ is canonically defined by the following property (using that there is an element $g_3 $ from $GL_n(\da_{\Delta})$
such that $g_3 \oo_{\da_{\Delta}}^n  \subset g_i \oo_{\da_{\Delta}}^n$, where $i=1$ and $i=2$). For any elements $g_1, g_2, g_3$ from $GL_n(\da_{\Delta})$ there is a canonical isomorphism of $\Z$-torsors
\begin{equation}   \label{dim3}
\Dim(g_1\oo_{\da_{\Delta}}^n \mid g_2\oo_{\da_{\Delta}}^n ) \otimes_{\Z}  \Dim(g_2\oo_{\da_{\Delta}}^n \mid g_3\oo_{\da_{\Delta}}^n )
\lrto \Dim(g_1\oo_{\da_{\Delta}}^n \mid g_3\oo_{\da_{\Delta}}^n )  \mbox{.}
\end{equation}
Any element $g$ from $GL_n(\da_{\Delta})$ defines an isomorphism of $\Z$-torsors for any elements $g_1, g_2$ from $GL_n(\da_{\Delta})$:
$$\Dim(g_1\oo_{\da_{\Delta}}^n \mid g_2\oo_{\da_{\Delta}}^n ) \lrto \Dim(gg_1\oo_{\da_{\Delta}}^n \mid gg_2\oo_{\da_{\Delta}}^n )
\quad
\mbox{,}
\qquad  \mbox{where} \quad
d \longmapsto g(d)  \mbox{.}
 $$

We obtain a central extension
\begin{equation}  \label{ext-1}
0 \lrto \Z \lrto \widetilde{GL_n(\da_{\Delta})}  \stackrel{\theta}{\lrto} GL_n(\da_{\Delta})  \lrto 1  \mbox{,}
\end{equation}
where the group $\widetilde{GL_n(\da_{\Delta})}$ is defined as the set of all pairs $(g,d)$, where $g \in GL_n(\da_{\Delta})$ and
$d \in \Dim(\oo_{\da_{\Delta}}^n \mid g\oo_{\da_{\Delta}}^n ) $,  with the multiplication law
given as
$$
(g_1, d_1)(g_2, d_2)= (g_1 g_2, d_1 \otimes g_1(d_2))  \mbox{,}
$$
and $\theta((g,d)) = g$.

 The following lemma is an important property which follows from the construction and formulas~\eqref{decomp}  and~\eqref{dim-ext} (compare also  with the proof of~\cite[Prop.~2]{Osi2}).
\begin{lemma}   \label{lem:Baer}
If $\Delta = \Delta_1 \cup \Delta_2$ such that $\Delta_1 \cap \Delta_2 = \emptyset$, then
the central extension  $\widetilde{GL_n(\da_{\Delta})}$ is the Baer sum (i.e. it corresponds to the sum of $2$-cocycles)
of central extensions $p_1^* \widetilde{GL_n(\da_{\Delta_1})}$ and $p_2^*\widetilde{GL_n(\da_{\Delta_1})} $, where $p_1$ and $p_2$
are projections in decomposition~\eqref{decomp}.
\end{lemma}

\subsection{Commutator of the lift  of elements and intersection index}   \label{sec_comm}
Using central extension~\eqref{ext-1} when $n=1$, for arbitrary elements $f,g$ from $\da_{\Delta}^*$ we define an element from $\Z$:
$$
\langle f,g \rangle_{\Delta}  \eqdef [\tilde{f},\tilde{g}]= \tilde{f}\tilde{g}\tilde{f}^{-1}\tilde{g}^{-1}   \mbox{,}
$$
where elements $\tilde{f}, \tilde{g}$ are from $\widetilde{GL_1(\da_{\Delta})}$ such that $\theta(\tilde{f})=f$ and $\theta(\tilde{g})=g$. The element $\langle f,g \rangle_{\Delta}$ does not depend on the choice of $\tilde{f}, \tilde{g}$.
The map $\langle \cdot, \cdot \rangle_{\Delta} $ is a bilinear and alternating map from $\da_{\Delta}^*  \times \da_{\Delta}^*$ to $\Z$.
From Lemma~\ref{lem:Baer} we have the following property (under conditions and notations of this lemma):
\begin{equation}  \label{dir_sum}
\langle f, g \rangle_{\Delta} =\langle p_1(f), p_1(g)   \rangle_{\Delta_1} +  \langle p_2(f), p_2(g)   \rangle_{\Delta_2}   \mbox{.}
\end{equation}
If $\Delta$ coincides with the set of all pairs $x \in C$ on $X$, then we will use also notation $\langle \cdot, \cdot \rangle_X$
for the map  $\langle \cdot, \cdot \rangle_{\Delta}$.

Let $K=k'((u))((t))$ be a two-dimensional local field, where  $k' \supset k$ is a finite extension of fields . By $\nu_K(\cdot, \cdot) : K^* \times K^* \to \Z $
we denote a bilinear and alternating map given as
\begin{equation}  \label{nu}
\nu_K(f,g) \eqdef  [k':k]  \cdot \nu_{\bar{K}} \left( \pi(f^{\nu_K(g)}{g^{-\nu_K(f)}}) \right) \mbox{,}
\end{equation}
where $f,g \in K^*$, the maps $\nu_K : K^* \to \Z$ and $\nu_{\bar{K}} : \bar{K}^*=k'((u))^*  \to \Z$ are discrete valuations,
and $\pi : \oo_K  \to \bar{K}$ is the natural homomorphism.
\begin{nt} {\em
There is another explicit formula for the expression $\nu_K(f,g)$  given as the product of the number $[k':k]$ and the determinant of $2 \times 2$-matrix of discrete valuations of rank $2$ for the elements $f$ and $g$.
See this and another properties of the map $\nu_K(\cdot, \cdot)$ in~\cite[\S~2.2]{Par} and, for example, in~\cite[\S~8.1]{GO}.
}
\end{nt}

For $K_{x,C} = \prod_{i=1}^l K_i$, where $K_i$ is a two-dimensional local field, we define a map $\nu_{x,C} : K_{x,C}^* \times K_{x,C}  \to \Z$ as
\begin{equation}  \label{nu-x-C}
\nu_{x,C}(f,g)  \eqdef \sum_{i=1}^l \nu_{K_i}(f_i, g_i )  \mbox{,}
\end{equation}
where $f,g $ are from $K_{x,C}^*$, and  $f_i, g_i$ are corresponding  projections of elements $f,g $ from $K_{x,C}^*$ to $K_i^*$.

\begin{prop}  \label{prop:symbol}
\begin{enumerate}
\item Let $\Delta$ be a single pair $x \in C$. In this case $\langle \cdot ,  \cdot  \rangle_{\Delta} = -\nu_{x,C}(\cdot, \cdot)$
\item For any set $\Delta$ of pairs $x \in C$ (as in the beginning of the paper) we have
\begin{equation} \label{sum}
\langle f ,  g  \rangle_{\Delta} = \sum_{ \{x\in C\} \in \Delta} \langle f_{x,C} ,  g_{x,C}  \rangle_{x \in C}  \mbox{,}
\end{equation}
where $f,g $ are from $ \da_{\Delta}^*$, the elements $f_{x,C}, g_{x,C}$ are corresponding  projections of elements $f,g  $ from  $\da_{\Delta}^*$
to $K_{x,C}^*$ (see formulas~\eqref{adel} and~\eqref{adel-delta}), and the sum in formula~\eqref{sum} contains only a finite number of non-zero terms.
\end{enumerate}
\end{prop}
\begin{proof}

1. Let $K_{x,C} = \prod_{i=1}^l K_i$, where $K_i$ is a two-dimensional local field. Using a direct analog of formula~\eqref{dir_sum} we reduce the statement to the following: $\langle f_i,g_i \rangle_{K_i} = -\nu_{K_i} (f_i, g_i)$, where the map $\langle \cdot, \cdot \rangle_{K_i}$ is constructed by the central extension which is  obtained as the  restriction  of the central extension $\widetilde{GL_1(K_{x,C})}$ from the group $GL_1(K_{x,C})$
to the subgroup $GL_1(K_i)$. Now this statement follows from Theorem~1 of~\cite{O}. (We note that there is a misprint with the sign in the statement and in the last line of the proof of Theorem~1 from~\cite{O}.)

2. Let  $\Delta_2$ be the set of all pairs $x \in C$ from $\Delta$ such that $f_{x,C} \oo_{K_{x,C}} = \oo_{K_{x,C}}$
and $g_{x,C} \oo_{K_{x,C}} = \oo_{K_{x,C}}$. Let $\Delta_1 $ be the complement set to $\Delta_1$ inside  the set $\Delta$.
By construction, the central extensions $\widetilde{GL_1(\da_{\Delta_2})}$ and $\widetilde{GL_1(K_{x,C})}$, where $\{x \in C\} \in \Delta_2$, split.
Therefore $\langle f_{\Delta_2} , g_{\Delta_2} \rangle_{\Delta_2} =0$, where $f_{\Delta_i}, g_{\Delta_i}$ ($i=1,2$) are corresponding projections of
elements $f,g$ from $\da_{\Delta}^*$ to $\da_{\Delta_i}^*$, and $\langle f_{x,C}, g_{x,C} \rangle_{x,C}=0 $ when $\{x \in C \} \in \Delta_2$.
Besides, from formulas~\eqref{nu} and~\eqref{nu-x-C} it follows that $\nu_{x,C}(f_{x,C}, g_{x,C})=0$ when $\{x \in C \} \in \Delta_2$.

Therefore from formula~\eqref{dir_sum} we obtain $\langle f ,  g  \rangle_{\Delta} = \langle f_{\Delta_1} ,  g_{\Delta_1}  \rangle_{\Delta_1}$. Thus we can change $\Delta$
to $\Delta_1$ in formula~\eqref{sum}. From construction of the set $\Delta_1$ we have that the  set of irreducible curves $C$  which appear in  pairs $x \in C$ from $\Delta_1$ is finite. Again by formula~\eqref{dir_sum} we can restrict ourself to a fixed irreducible curve $C$, i.e. we consider  a set $\Delta$ such that a curve $C$ is fixed for pairs $x \in C$ from $\Delta$.

Since $f \in \da_{\Delta}$ and $f^{-1} \in \da_{\Delta}$, from adelic conditions we obtain that there is a finite set of  integers such that
$\nu_{K_{x,C}}(f_{x,C})$ belongs to this set when $x$ runs over all smooth points on $C$ from pairs $\{x \in C \} \in \Delta$. (If $x$ is a smooth point on $C$, then $K_{x,C}$ is a two-dimensional local field with the discrete valuation $\nu_{K_{x,C}}$.) The same is true for the element $g \in \da_{\Delta}^*$, but with possibly another finite set. Therefore, subdividing the set $\Delta$ into a finite number of subsets and using formula~\eqref{dir_sum} we will suppose that $\Delta$  satisfies conditions of one of   the following two cases.
 In the former case, the set $\Delta$ consists of one pair $x \in C$ (when $x$ is a singular point on $C$), and therefore  formula~\eqref{sum}
 is tautological and we will not consider this case further. In the remaining case,
   the integers $\nu_{K_{x,C}}(f_{x,C})$  and  $\nu_{K_{x,C}}(g_{x,C})$ do not change when   $x$  runs over all smooth points on $C$ such that $\{ x \in C \} \in \Delta $.

Let $t_C=0$ be  an equation of the irreducible curve $C$ on some open subset of $X$. Then using bilinear and alternating property of both hand sides of formula~\eqref{sum}, and also the above properties of the set $\Delta$, we obtain that it is enough to consider two cases: 1) $f$ and $g$
are from $\oo_{\da_{\Delta}}^*$; 2) $f \in \oo_{\da_{\Delta}}^*$ and $g = t_C$. In the first case, $f_{x,C}$ and $g_{x,C}$ are from $\oo_{K_{x,C}}^*$
for all pairs $x \in C$ from $\Delta$. Therefore, by construction, central extensions $\widetilde{GL_1(\da_{\Delta})}$ and $\widetilde{GL_1(K_{x,C})}$, where $\{x \in C \} \in \Delta $, split. Hence $\langle f , g \rangle_{\Delta} =0$ and
$\langle f_{x,C}, g_{x,C} \rangle_{x,C}=0 $ when $\{x \in C \} \in \Delta$, and formula~\eqref{sum} follows. In the second case,
the right hand side of formula~\eqref{sum} equals to $\sum_{\{x \in C \} \in  \Delta} - \nu_{x,C} (f_{x,C}, t_C)$ by the first statement of this proposition, and this sum contains only a finite number of non-zero terms by formulas~\eqref{nu} and~\eqref{nu-x-C}  and adelic conditions on $f$.
On the other hand, by definition of $\langle \cdot, \cdot \rangle_{\Delta}$ we have   $\langle f, t_C^{-1}  \rangle_{\Delta} =  d(\pi(f)^{-1} (U) - d(U)$, where $\pi$ is the natural homomorphism $\oo_{\da_{\Delta}}  \to \oo_{\da_{\Delta}}/ t_C \oo_{\da_{\Delta}}$,  $d$   is a dimension theory  on $\oo_{\da_{\Delta}}/ t_C \oo_{\da_{\Delta}}$ and $U$ is an open linearly compact $k$-subspace in  $\oo_{\da_{\Delta}}/ t_C \oo_{\da_{\Delta}}$. (Compare with the calculation of case 2 in the proof of Theorem~1 of~\cite{O}.) Fixing  an open set $U$ as the product of rings of integers of one-dimensional local fields, and   dimension theory $d$ such that $d(U)=0$, it is easy to see that
$ d(\pi(f)^{-1} (U) - d(U) = \sum _{\{x \in C \} \in  \Delta}  \nu_{x,C} (f_{x,C}, t_C)$. Thus we obtain formula~\eqref{sum} in this case.
\end{proof}

\medskip

For a surface $X$, an irreducible curve $C \subset X$, and  a point $x \in X$, let $K_C$ be the  completion
of the field  $k(X)$ of rational functions on $X$ with respect to the discrete valuation given by $C$, let $K_x = k(X) \cdot \hat{\oo}_{x,X}$ be a subring of the fraction field  $\Frac \hat{\oo}_{x,X}$.

Let $D$ be a divisor on $X$.

For an irreducible curve $C \subset X$ let $j^D_C \in K_C^*$ be an equation of the divisor $D$ after the restriction to $\Spec K_C$. For any point $y \in C$ we have an embedding $K_C  \subset K_{y,C}$. It is easy to check that a collection $\{ j^D_C \}$, where $C$ runs over the set of all irreducible curves on $ X$, defines a well-defined element from $\da_X^*$ under the natural diagonal embedding $\prod_{C \subset X} K_C \hookrightarrow \prod_{y \in C } K_{y,C }$.

For a point $x \in X$ let $j^D_x \in K_x^*$ be an equation of the divisor $D$ after the restriction to $\Spec K_x$. For any irreducible curve $E \ni x$ we have an embedding $K_x \subset K_{x,E}$. It is easy to check that a collection $\{ j^D_x \}$, where $x$ runs over the set of all  points of $ X$, defines a well-defined element from $\da_X^*$ under the natural diagonal embedding
${\prod_{x \in X} K_x \hookrightarrow \prod_{x \in E } K_{x,E }}$.

Using the definition of the intersection index of divisors given by A.N.~Parshin in~\cite[\S~2.2]{Par}  by means of sum of local maps $\nu_{x,C}$, we immediately obtain from Proposition~\ref{prop:symbol} the following proposition. (We note that the analog of this proposition was used without written proof in~\cite{OsiPa}.)

\begin{prop} \label{intersect} Let $S$ and $T$ be divisors on a smooth projective surface $X$, and ${(S,T) \in \Z}$ be their intersection index. We have
$$
\langle \{j^S_x\}, \{ j^T_C  \}   \rangle_X = - (S,T)  \mbox{.}
$$
\end{prop}

\section{Second Chern numbers}
\subsection{Central extension $\widehat{GL_n(\da_{\Delta})}$}  \label{centr-ext2}

For any $\Delta$ which is a subset of all pairs $x \in C$, where $C$ is an irreducible curve on  $X$. We have natural isomorphism of groups
$$
GL_n(\da_{\Delta}) = SL_n(\da_{\Delta}) \rtimes \da_{\Delta}^*  \mbox{,}
$$
where the group $\da_{\Delta}^*  $ is embedded into the
upper left corner of the group $GL_n(\da_{\Delta})$ and acts on the group $SL_n(\da_{\Delta})$ by conjugation, i.e.  by inner automorphisms  $h \mapsto aha^{-1}$,
where $a \in \da_{\Delta}^*$ and $h \in SL_n(\da_{\Delta})$.
By means of the central extension~\eqref{ext-1}  the action of the group $\da_{\Delta}^*$ is  lifted to the action  on the group $\theta^{-1}(SL_n(\da_{\Delta}))$ (by inner automorphisms  of the group $\widetilde{GL_n(\da_{\Delta})}$ .
We define a group $$ \widehat{GL_n(\da_{\Delta})} \eqdef \theta^{-1}(SL_n(\da_{\Delta})) \rtimes \da_{\Delta}^* \mbox{,}$$
whose natural homomorphism to $GL_n(\da_{\Delta})$ gives a central extension
$$
0 \lrto \Z \lrto \widehat{GL_n(\da_{\Delta})} {\lrto} GL_n(\da_{\Delta})  \lrto 1  \mbox{,}
$$
which, by construction, splits over the subgroup $\da_{\Delta}^*$ of $GL_n(\da_{\Delta})$.

\begin{nt}\em
To construct central extension $\widehat{GL_n(\da_{\Delta})}$ we used an embedding of $\da_{\Delta}^*$ to $GL_n(\da_{\Delta})$ as
$a \mapsto \diag(a,1, \ldots, 1)$, where $a \in \da_{\Delta}^*$. Since an inner automorphism of the group $GL_n(\da_{\Delta})$
induces a canonical automorphism of  the group that is a central extension of $GL_n(\da_{\Delta})$, another embedding $a \mapsto \diag(1, \ldots, a, \ldots, 1)$
of $\da_{\Delta}^*$ to $GL_n$
(into other place on the diagonal)
produces a construction of the central extension which is canonically isomorphic to the central extension  $\widehat{GL_n(\da_{\Delta})}$ (compare also with Remark~3 from~\cite{Osi2}).
\end{nt}

\begin{nt}  \label{dec-hat} \em
From formula~\eqref{decomp}  and Lemma~\ref{lem:Baer} we obtain the property which is similar to the statement of Lemma~\ref{lem:Baer} when we replace the central extensions  $\widetilde{GL_n(\da_{\Delta})}$, $\widetilde{GL_n(\da_{\Delta_1})}$  and $\widetilde{GL_n(\da_{\Delta_2})}$ to $\widehat{GL_n(\da_{\Delta})}$, $\widehat{GL_n(\da_{\Delta_1})}$  and $\widehat{GL_n(\da_{\Delta_2})}$ correspondingly
(compare with~\cite[Prop.~2]{Osi2}).
\end{nt}

\bigskip

 The analogy with the next proposition (and with remark after them) can be found in~\cite[\S~A5]{BSh} and~\cite[Appendix]{CPT}, where it was considered a central extension of a group $GL_n(A)$ by a group $K_2(A)$ for a ring $A$ with the property $SK_1(A)=0$. We note
that it is not clear how to deduce the next proposition (and remark after them) from~~\cite[\S~A5]{BSh} and~\cite[Appendix]{CPT}.

We consider a central extension
\begin{equation}  \label{double}
0 \lrto \Z \lrto  \widehat{\da_{\Delta}^*  \times \da_{\Delta}^* }  \lrto \da_{\Delta}^*  \times \da_{\Delta}^* \lrto 1  \mbox{,}
\end{equation}
where $\widehat{\da_{\Delta}^*  \times \da_{\Delta}^* } \eqdef \da_{\Delta}^*  \times \da_{\Delta}^*  \times \Z $ as a set, and with the multiplication law given as
$$(f,g;r) (f',g';r') \eqdef (ff', gg'; r+r' + \langle f',g   \rangle_{\Delta}) \mbox{,} $$
where $f,g,f',g'$ are from $\da_{\Delta}^*$, and $r, r'$ are from $\Z$.

For any $a \in \da_{\Delta}^*$ we denote by $\phi_1(a)$ the element from $\widehat{GL_n(\da_{\Delta})}$ which equals to the canonical section of the central extension $\widehat{GL_n(\da_{\Delta})}$ over the subgroup $\da_{\Delta}^*$ applied to the element $a$.
For any integer $l$ such that $1 \le l \le n$  we denote $\phi_l(a) = \Phi_l a \Phi_l^{-1}$,
 where $\Phi_l$ is a lift to $\widehat{GL_n(\da_{\Delta})}$ of the matrix from $GL_n(\da_{\Delta})$ which acts as transposition on standard  coordinates of $\da_{\Delta}^n$ permuting the first and the $l$-th coordinates. Clearly, $\phi_l(a)$
 does not depend on a lift of such matrix, and the image of $\phi_l(a)$ under the standard homomorphism to $GL_n(\da_{\Delta})$ equals to
 $\diag(1, \ldots, a , \ldots, 1)$ with the element $a$ is located on the $l$-th place of the diagonal.

\begin{prop}  \label{ext-restr}
\begin{enumerate}
\item We fix integers $1 \le i < j \le n$ and embed the group $\da_{\Delta}^*  \times \da_{\Delta}^*$ into the group $GL_n(\da_{\Delta})$
as $(f,g)  \mapsto \diag(1, \ldots, f, \ldots, g, \ldots, 1)$, where  elements $f$ and $g$ from $\da_{\Delta}^*$
are located on $i$-th and $j$-th places on the diagonal. We obtain that the restriction of the central extension $\widehat{GL_n(\da_{\Delta})}$
to the subgroup $\da_{\Delta}^*  \times \da_{\Delta}^*$ is isomorphic to the central extension~\eqref{double} via the map
$$
r \phi_i(f)  \phi_j(g)  \longmapsto  (f,g; r)  \mbox{,}
$$
where $r $ is from $\Z$, which is a subgroup of  the center of the group $\widehat{GL_n(\da_{\Delta})}$.
\item
For positive integers $n_1$ and $n_2$ such that  $n = n_1 + n_2$ we consider a subgroup
$$
P_{n_1, n_2}(\da_{\Delta}) \eqdef \left\{
\begin{pmatrix}
GL_{n_1}(\da_{\Delta})   & * \\
0 & GL_{n_2}(\da_{\Delta})
\end{pmatrix}
\right\}
\subset GL_n(\da_{\Delta})
$$
Let $p_i : P_{n_1, n_2}(\da_{\Delta})  \to GL_{n_i} (\da_{\Delta})$ be the projections, where $i=1$ and $i=2$. We obtain that the restriction
 of the central extension $\widehat{GL_n(\da_{\Delta})}$
to the subgroup $P_{n_1,n_2}(\da_{\Delta})$ is isomorphic to the Baer sum of central extensions $p_1^* (\widehat{GL_{n_1}(\da_{\Delta})})$,
$p_2^* (\widehat{GL_{n_2}(\da_{\Delta})}  )$ and $ (\det(p_1)  \times \det(p_2))^* \widehat{\da_{\Delta}^*  \times \da_{\Delta}^* } $.
\end{enumerate}
\end{prop}
\begin{proof}
1. For any element $a \in \da^*$ and integer $i$ such that  $1 \le i \le n$, we denote $d_i(a) = \diag(1, \ldots, a, \ldots, 1) \in GL_n(\da_{\Delta})$,
 where $a$ is located on $i$-th place in the diagonal matrix.

It is enough to prove the following equality inside the group $\widehat{GL_n(\da_{\Delta})}$ for any elements  $f,g,f',g'$ from $GL_n(\da_{\Delta})$:
$$
\phi_i(f)  \phi_j(g) \phi_i(f') \phi_j(g') = \langle f', g \rangle_{\Delta} \cdot \phi_i(ff') \phi_j(gg') \mbox{.}
$$
Clearly, this equality follows from an equality  $\phi_j(g) \phi_i(f') = \langle f', g \rangle_{\Delta} \cdot \phi_i(f') \phi_j(g)$.
Thus, we have to prove that $[\phi_j(g), \phi_i(f')]= \langle f', g \rangle_{\Delta}$
 or $[\phi_i(f'), \phi_j(g)]= \langle g, f' \rangle_{\Delta}$.
 Applying the conjugation by  $\Phi_i$, we obtain that it is enough to prove an equality ${[\phi_1(f'), \Phi_i \cdot \phi_j(g)  \cdot \Phi_i^{-1}   ] = \langle  g, f' \rangle_{\Delta}} $. We note that the image of $\Phi_i \cdot \phi_j(g)  \cdot \Phi_i^{-1}$
 under the standard homomorphism to $GL_n(\da_{\Delta})$ equals to
 $d_j(g)$.  Since the commutator of lifts of two elements does not depend on the choice of lifts of these elements  to the central extension, we have that
 ${[\phi_1(f'), \Phi_i \cdot \phi_j(g)  \cdot \Phi_i^{-1}   ] = [\phi_1(f'), \phi_j(g)]}$.  Further,
  using the bilinear property of the commutator of lifts of commuting elements,
  we obtain
 \begin{multline}  \label{comm_mult}
 [\phi_1(f'), \phi_j(g)]= [\phi_1(f'), \phi_1(g) \phi_1(g)^{-1} \phi_j(g)]= \\ = [\phi_1(f'), \phi_1(g)] \cdot [\phi_1(f'), \phi_1(g)^{-1} \phi_j(g) ] =
 [\phi_1(f'), \phi_1(g)^{-1} \phi_j(g) ] \mbox{.}
 \end{multline}
We denote by $h$ the image   of the element $ \phi_1(g)^{-1} \phi_j(g)$ under the homomorphism to $GL_n(\da_{\Delta})$. Since $h$ belongs to the subgroup $SL_n(\da_{\Delta})$, by construction of the group  $\widehat{GL_n(\da_{\Delta})}$ the last commutator in formula~\eqref{comm_mult} equals to commutator
$[\widetilde{d_1(f')}, \tilde{h}]$ computed in the group  $\widetilde{GL_n(\da_{\Delta})}$, where $\widetilde{d_1(f')}$
and $\widetilde{h}$ are lifts of elements $d_1(f')$ and $h$ from the group $GL_n(\da_{\Delta})$ to the group $\widetilde{GL_n(\da_{\Delta})}$.  We obtain  in the group $\widetilde{GL_n(\da_{\Delta})}$  an equality
$$
[\widetilde{d_1(f')}, \tilde{h}]= [\widetilde{d_1(f')}, \widetilde{d_1(g^{-1})}  \cdot \widetilde{d_j(g)}] =
[\widetilde{d_1(f')}, \widetilde{d_1(g^{-1})}] \cdot [\widetilde{d_1(f')},  \widetilde{d_j(g)}] \mbox{.}
$$
Now from construction of the group   $\widetilde{GL_n(\da_{\Delta})}$ we obtain that
the commutator of the lift of diagonal matrices can be calculated
componentwise (first, separately for each place on the diagonal, and then multiply the results). Therefore we obtain
$[\widetilde{d_1(f')}, \widetilde{d_1(g^{-1})}] = \langle g, f'\rangle_{\Delta}$,
and  since $j\ne1$, we have  $[\widetilde{d_1(f')},  \widetilde{d_j(g)}]=1$.

2. We embed the group $\da_{\Delta}^*  \times \da_{\Delta}^*$ into the group $P_{n_1, n_2}(\da_{\Delta})$
in the following way: ${(f,g)  \mapsto \diag(f, \ldots, g, \ldots, 1)}$, where  elements $f$ and $g$ from $\da_{\Delta}^*$
are located on the first and $(n_1+1)$-th places on the diagonal, and other places of the diagonal are occupied by~$1$.  We can write the group $P_{n_1, n_2} (\da_{\Delta})$ as a semidirect product:
\begin{equation}  \label{semi}
\left\{
\begin{pmatrix}
SL_{n_1}(\da_{\Delta})   & * \\
0 & SL_{n_2}(\da_{\Delta})
\end{pmatrix}
\right\}
\rtimes
(\da_{\Delta}^*  \times \da_{\Delta}^*)  \mbox{.}
\end{equation}

According to Construction~1.7 from~\cite{BrD}, a central extension $\widehat{G \rtimes H}$ of a semidirect product $G \rtimes H$ by a group $A$ is equivalent to the following data: $1)$ a central extension $\hat{G}$ of $G$ by $A$; $2)$ a central extension $\hat{H}$ of $H$ by $A$; $3)$ an action of $H$ on $\hat{G}  \to G$, lifting the action of $H$ on $G$. We note that central extensions  $\hat{G}$ and $\hat{H}$ are obtained as restrictions of the central extension $\widehat{G \rtimes H}$ to the subgroups $G$ and $H$ correspondingly.

To prove the second statement of Proposition~\ref{ext-restr}, we apply the above construction to the  case of semidirect product given by formula~\eqref{semi}, where $H = \da_{\Delta}^*  \times \da_{\Delta}^*$ and $G = \left\{
\begin{pmatrix}
SL_{n_1}(\da_{\Delta})   & * \\
0 & SL_{n_2}(\da_{\Delta})
\end{pmatrix}
\right\}$. Then condition $2)$ of the construction follows from the first statement of Proposition~\ref{ext-restr}. To obtain condition $1)$ we note that
the restriction of the central extension $\widetilde{GL_n(\da_{\Delta})}$ to the subgroup $P_{n_1, n_2}(\da_{\Delta})$
is canonically isomorphic to the Baer sum of central extensions $p_1^* \widetilde{GL_{n_1}(\da_{\Delta})}$ and
$p_2^* \widetilde{GL_{n_2}(\da_{\Delta})}$. This fact easily follows from a natural action of $ P_{n_1, n_2}(\da_{\Delta})$ on an exact triple (exact triple of complete $C_2$-vector spaces over $k$)
\begin{equation}  \label{da-exact}
0 \lrto \da_{\Delta}^{n_1}  \lrto \da_{\Delta}^{n}  \lrto \da_{\Delta}^{n_2} \lrto 0 \mbox{,}
\end{equation}
from formula~\eqref{dim-ext},  and from the construction of the group $\widetilde{GL_n(\da_{\Delta})}$. To finish the checking of condition $2)$ we note that the restrictions of the central extensions $\widetilde{GL_n(\da_{\Delta})}$ and $\widehat{GL_n(\da_{\Delta})}$ to the subgroup $SL_n(\da_{\Delta})$  coincide (or canonically isomorphic). To obtain condition 3) we note that an action of the group $H$ on $\hat{G} \to G$ comes from the action by conjugations of elements from $\da_{\Delta}^* \times \da_{\Delta}^*$ (lifted to $\widetilde{GL_n(\da_{\Delta})}$) on $\widetilde{GL_n(\da_{\Delta})}$ restricted as central extension to $P_{n_1, n_2}(\da_{\Delta})$. Besides, it is important that the group $H$ naturally acts on exact triple~\eqref{da-exact}, and therefore the action of $H$ on
$\hat{G}$ is compatible with the action on the corresponding Baer sum with respect to projections $p_1$ and $p_2$.
\end{proof}

\begin{nt} {\em
\begin{enumerate}
\item
From the first statement of Proposition~\ref{ext-restr} it is easy to obtain the following generalization.
For any integer $k$ such that $1 \le k \le n $ we consider a central extension
\begin{equation}  \label{k-tuple}
0 \lrto \Z \lrto  \widehat{({\da_{\Delta}^*})^k}  \lrto ({\da_{\Delta}^*})^k \lrto 1  \mbox{,}
\end{equation}
where
$({\da_{\Delta}^*})^k$ is the direct product $\da_{\Delta}^*  \times \ldots \da_{\Delta}^*$ with $\da_{\Delta}^*$ beeing taken $k$ times, and
$\widehat{({\da_{\Delta}^*})^k} \eqdef ({\da_{\Delta}^*})^k \times \Z$ as a set, where the group multiplication law is given as
$$
(f_1, \ldots, f_k; r)(f_1', \ldots, f_k'; r') \eqdef (f_1 f_1', \ldots, f_k f_k'; r + r' + \sum_{i < j} \langle f_i', f_j \rangle_{\Delta} )  \mbox{,}
$$
where $f_1, \ldots, f_k, f_1',  \ldots, f_k'$ are from $\da_{\Delta}^*$, and $r, r'$ are from $\Z$. We fix integers $1 \le j_1 < \ldots j_k \le n$,
and consider an embedding of the group $({\da_{\Delta}^*})^k$ to the group $GL_n(\da_{\Delta})$ given as: $(f_1, \ldots, f_k)$ is mapped to the diagonal matrix $\diag(a_1, \ldots, a_n)$, where $a_{j_i}=f_i$ (for $1 \le i \le k$), and $a_l=1$ otherwise. Then the restriction of the central extension  $\widehat{GL_n(\da_{\Delta})}$ to the subgroup  $({\da_{\Delta}^*})^k$ is isomorphic to the central extension~\eqref{k-tuple} via the map
$$
r \phi_{j_1}(f_1)  \cdot \ldots \cdot  \phi_{j_k}(f_k)  \longmapsto  (f_1, \ldots, f_k; r)  \mbox{.}
$$
\item
The central extension  $\widehat{GL_n(\da_{\Delta})}$
canonically splits over $\da_{\Delta}^*$, where this group is embedded into the $i$-th place of the diagonal, via the map $a \mapsto \phi_i(a)$. From this fact and  the second statement of Proposition~\ref{ext-restr} we obtain that the central extension
$\widehat{GL_n(\da_{\Delta})}$  canonically splits over the subgroup
$ U_n=
\left\{
\begin{pmatrix}
1  & & * \\
& 1 & \\
0 & & 1
\end{pmatrix}
\right\}
$.
\end{enumerate}
}
\end{nt}

\subsection{Canonical splittings}   \label{can-split}

Now we give the generalization of non-commutative reciprocity laws from~\cite[\S~3.5]{Osi2}    when $X$ is a smooth algebraic surface over $k$.

We recall (see the end of Section~\ref{sec_comm}) that we have diagonal embeddings
$$
\prod_{C \subset X} K_{C} \hookrightarrow \prod_{x \in C} K_{x,C}     \qquad  \mbox{and} \qquad \prod_{x \in X} K_x \hookrightarrow \prod_{x \in C} K_{x,C}  \mbox{.}
$$
There are the following subrings of the adelic ring $\da_X$:
\begin{equation}  \label{subrings}
\da_{X,01} = (\prod_{C \subset X} K_{C})  \cap \da_X    \quad \mbox{,}  \qquad     \qquad \da_{X, 02} =  (\prod_{x \in X}  K_x)    \cap \da_X
\quad   \mbox{,}  \qquad \mbox{and} \qquad \da_{X, 12}= \oo_{\da_X}
\mbox{,}
\end{equation}
where the intersection is taken inside the ring $\prod_{x \in C} K_{x,C}$.

\begin{prop}  \label{split-ext}
\begin{enumerate}
\item \label{st1}
For any set $\Delta$ of pairs $x \in C$ the central extensions $\widetilde{GL_n(\da_{\Delta})}$ and $\widehat{GL_n(\da_{\Delta}) }$
canonically split over the subgroup $GL_n(\oo_{\da_{\Delta}})$ of the group $GL_n(\da_{\Delta})$.
\item \label{st2}
The central extensions $\widetilde{GL_n(\da_{X})}$ and $\widehat{GL_n(\da_{X}) }$
canonically split over the subgroup $GL_n(\da_{X, 02})$ of the group $GL_n(\da_{X})$.
\item \label{st3}
Suppose that $X$ is projective. Then the central extensions $\widetilde{GL_n(\da_{X})}$ and $\widehat{GL_n(\da_{X}) }$
canonically split over the subgroup $GL_n(\da_{X, 01})$ of the group $GL_n(\da_{X})$.
\item \label{st4}
Splittings of the central extension
 $\widehat{GL_n(\da_{X}) }$ from statements~\ref{st2}-\ref{st3}  coincide over the  subgroup $GL_n(k(X))$.
 The analogous results are also true for the splittings from statements~\ref{st1}-\ref{st2} and the subgroup $GL_n(\da_{X,12} \cap \da_{X, 02})$, and for the splittings from statements~\ref{st1}, \ref{st3} and the subgroup $GL_n(\da_{X, 12} \cap \da_{X, 01})$.
 \item  \label{st5}
Under the same conditions as in statements~\ref{st1}-\ref{st3}, the central extension $\widehat{\da_{\Delta}^* \times \da_{\Delta}^*}$ (see formula~\eqref{double})
splits over the subgroups $\oo_{\da_{\Delta}}^* \times \oo_{\da_{\Delta}}^*$, $\da_{X,02}^* \times \da_{X, 02}^*$ and
$\da_{X, 01}^* \times \da_{X, 01}^*$ via the map $(f,g) \mapsto (f,g;0)$.
\item   \label{st6}
For the central extension $\widehat{GL_n(\da_{X}) }$ restricted to a subgroup ${P_{n_1, n_2} (\da_{X}) \subset GL_n(\da_X)}$,  splittings from statements~\ref{st1}-\ref{st3} and~\ref{st5} are compatible with respect to the isomorphism constructed in  the second statement of Proposition~\ref{ext-restr}.
\end{enumerate}
\end{prop}
\begin{proof}
1. The splittings follow from the constructions of the central extensions  $\widetilde{GL_n(\da_{\Delta})}$ and $\widehat{GL_n(\da_{\Delta}) }$,
since for any element $f \in GL_n(\oo_{\da_{\Delta}})$ we have $f \oo_{\da_{\Delta}}^n = \oo_{\da_{\Delta}}^n$.

2. First we prove that the central extension $\widetilde{GL_n(\da_{X})}$ splits over the subgroup $GL_n(\da_{X,02})$.
We note that for any two divisors $D_1 \ge D_2$ the subspace
$$(\oo_{\da_X}(D_2) \cap \da_{X, 02})  / (\oo_{\da_X}(D_1) \cap \da_{X,02}) \quad \subset \quad \oo_{\da_X}(D_2)   / \oo_{\da_X}(D_1) $$
is an open linearly compact $k$-vector space. Hence for any $g_1, g_2  \in GL_n(\da_X)$ such that ${g_2 \oo_{\da_X}^n   \supset g_1 \oo_{\da_X}^n}$
the subspace
$$(g_2 \oo_{\da_X}^n \cap \da_{X, 02}^n)  / ( g_1 \oo_{\da_X}^n \cap \da_{X,02}^n) \quad \subset \quad g_2 \oo_{\da_X}^n   / g_1 \oo_{\da_X}^n $$
 is an open linearly compact $k$-vector space. We define $d_{g_1, g_2} \in \Dim( g_2 \oo_{\da_X}^n   / g_1 \oo_{\da_X}^n)$ by the rule
 $d_{g_1, g_2} ((g_2 \oo_{\da_X}^n \cap \da_{X, 02}^n)  / ( g_1 \oo_{\da_X}^n \cap \da_{X,02}^n))=0 $. Using formulas~\eqref{dim1}-\eqref{dim3},
 we obtain a well-defined element $d_{g_1,g_2} \in \Dim ( g_1 \oo_{\da_X}^n   \mid g_2 \oo_{\da_X}^n)$ for any elements $g_1, g_2 \in GL_n(\da_X)$.
 Now it is easy to see that the map
 \begin{equation}  \label{splitt}
  GL_n(\da_{X, 02}) \lrto \widetilde{GL_n(\da_{\Delta})} : \quad  g \longmapsto (g, d_{1,g})
  \end{equation}
 is a group splitting.

 To prove the splitting of  the central extension $\widehat{GL_n(\da_{X})}$ over the subgroup $GL_n(\da_{X,02})$ we note that
 $GL(n, \da_{X, 02}) = SL(n, \da_{X, 02}) \rtimes \da_{X, 02}^* $. For any $a \in \da_{X, 02}^* $, the conjugation by the element
  $(a, d_{1,a})  \in \widetilde{GL_n(\da_{\Delta})}$  does not change the section over the group $SL(n, \da_{X, 02})$ which was constructed in formula~\eqref{splitt}. By construction, this gives the splitting of $\widehat{GL_n(\da_{X})}$ over $GL_n(\da_{X,02})$, where we take the trivial section over $\da_{X, 02}^*$.

3. The idea for the proof of this statement is analogous to the proof of statement~\ref{st2}, but instead  of element
$d_{g_1,g_2}  \in \Dim ( g_1 \oo_{\da_X}^n   \mid g_2 \oo_{\da_X}^n)$ we have to use another element $d_{g_1,g_2}^{\prime}  \in \Dim ( g_1 \oo_{\da_X}^n   \mid g_2 \oo_{\da_X}^n)$ which is constructed by means of the following property. For any two divisors $D_1 \ge D_2$ the subspace
$U =(\oo_{\da_X}(D_2) \cap \da_{X, 01})  / (\oo_{\da_X}(D_1) \cap \da_{X,01})$ of the space  $V=\oo_{\da_X}(D_2)   / \oo_{\da_X}(D_1) $
is a discrete $k$-vector subspace such that $V/U$ is a  linearly compact $k$-vector space.  (It is important that on the projective curve $C$ the field of rational functions $k(C)$  is a discrete subspace inside the adelic ring of $C$, and the quotient space is a linearly compact $k$-vector space that follows, for example, from the adelic complex on the curve $C$ and the fact that $k$-vector spaces $H^0(C, \oo_C)$ and $H^1(C, \oo_C)$ are finite-dimensional over $k$.) Hence for any $g_1, g_2  \in GL_n(\da_X)$ such that ${g_2 \oo_{\da_X}^n   \supset g_1 \oo_{\da_X}^n}$
we have an exact triple of $k$-vector spaces
$$
0 \lrto Y   \lrto W   \lrto W/Y   \lrto 0   \mbox{,}
$$
where $W= g_2 \oo_{\da_X}^n   / g_1 \oo_{\da_X}^n$ is a locally linearly  compact $k$-vector space, $k$-vector space ${Y=   (g_2 \oo_{\da_X}^n \cap \da_{X, 01}^n)  / ( g_1 \oo_{\da_X}^n \cap \da_{X,01}^n)}$
is a discrete subspace in induced topology, and the space $W/Y$ endowed with the quotient topology is a linearly compact $k$-vector space.
Using formulas~\eqref{ex_seq}-\eqref{dim-ext}   we define an element $d_{g_1, g_2}'  \in \Dim(W) $ as $d_Y \otimes d_{W/Y}$, where $d_Y \in \Dim(Y)$ is defined as $d_Y((0))=0$ (here $(0)$ is the zero subspace), and $d_{W/Y} \in \Dim (W/Y)$ is defined as $d_{W/Y}(W/Y)=0$. To finish we proceed further as in the proof of statement~\ref{st2}.

4. The group $SL_n(k(X))$ is perfect. Therefore any two sections of the central extension  $\widetilde{GL_n(\da_{X})}$ restricted to the group $SL_n(k(X))$ coincide. Hence,  two sections of the central extension $\widehat{GL_n(\da_{X})}$ restricted to the group $GL_n(k(X))$ coincide, because over the subgroup $k(X)^* : a \mapsto \diag(a,1, \ldots, 1)$ two sections are trivial by constructions in the proof of statements~\ref{st2}-\ref{st3}.

Various splittings of  the central extension $\widetilde{GL_n(\da_{X})}$  over the subgroup \linebreak ${GL_n(\da_{X, 01} \cap \da_{X,12}) }$
and over the subgroup ${GL_n(\da_{X, 02} \cap \da_{X,12}) }$ coincide, because
 for any element $f \in GL_n(\oo_{\da_{\Delta}})$ we have $f \oo_{\da_{\Delta}}^n = \oo_{\da_{\Delta}}^n$, and then we have to use the constructions of the splittings.
Hence, again by construction,  the same is true for the central extension $\widehat{GL_n(\da_{X})}$.

5. By statements~\ref{st1}-\ref{st3} the central extension $\widetilde{GL_1(\da_{\Delta})}$ splits over $\oo_{\da_{\Delta}}^*$, the central extension
$\widetilde{GL_1(\da_{X})}$ splits over $\da_{X, 02}^*$ and over $\da_{X, 01}^*$ (when $X$ is projective). Hence $\langle f,g  \rangle_{\Delta}=0$
for $f,g \in \oo_{\da_{\Delta}}^*$, and $\langle f, g  \rangle_{X} =0  $ for $f,g \in \da_{X,02}^*$ and when $X$ is projective for $f,g \in \da_{X, 01}^*$. Now we finish by the definition of the multiplication law in $\widehat{\da_{\Delta}^* \times \da_{\Delta}^*}$.

6. This statement follows from constructions of sections in proofs of statements~\ref{st1}-\ref{st3} and of the second statement  of Proposition~\ref{ext-restr}. It is important that together with exact sequence~\eqref{da-exact} we can write exact sequences
$$
0 \lrto \da_{X, 02}^{n_1}  \lrto \da_{X, 02}^n    \lrto \da_{X, 02}^{n_2}  \lrto 0
\quad \mbox{and}   \quad
0 \lrto \da_{X, 01}^{n_1}  \lrto \da_{X, 01}^n    \lrto \da_{X, 01}^{n_2}  \lrto 0
$$
and the corresponding groups $ P_{n_1, n_2} (\da_{X}) \cap GL_n(\da_{X, 02}) $ and
$ P_{n_1, n_2} (\da_X) \cap GL_n(\da_{X, 01})$ act on these sequences.

Besides, concerning the central extension $\widehat{\da_X^* \times \da_X^*}$, we note that for any integer $l$ such that $1 \le l \le n $
 if an element  $a$ is from $ \da_{X, ij}^*$ with $ij$ equal to $12$ or $02$ or $01$, then the element $\phi_l(a) \in \widehat{GL_n(\da_{X})}$ (see its definition before Proposition~\ref{ext-restr}) equals to a  section
over the element $\diag(1, \ldots, a, \ldots, 1)$, where this section is constructed in statements~\ref{st1}, \ref{st2} or \ref{st3} correspondingly,
and $a$ is located on the $l$-th place of the diagonal in $\diag(1, \ldots, a \ldots, 1)$.
This is because the matrix of transposition of coordinates which was used to construct $\phi_l(a)$ belongs to any of the groups: $GL_n (\da_{X,12})$,
$GL_n(\da_{X, 02})$ and $GL_n(\da_{X, 01})$. Hence, the map $\diag(f, \ldots, g, \ldots, 1) \mapsto \phi_1(f) \phi_{n_1+1}(g)$
equals to a section from the proof of statements~\ref{st1}, \ref{st2} and~\ref{st3} when $f,g  \in \da_{X, ij}^*$, where $ij$  equal to $12$ or $02$
or $01$ correspondingly (compare with the proof of the second statement of Proposition~\ref{ext-restr}).
\end{proof}

\begin{nt}{\em
For a smooth projective surface $X$ we have $\da_{X, 02}  \cap \da_{X, 01} = k(X)$, see~\cite[Th.~6(iii)]{BG} (after T.~Fimmel and A.N.~Parshin). Therefore we obtain \linebreak ${GL_n(k(X))= GL_n(\da_{X,02})  \cap GL_n(\da_{X, 01})  }$. Hence we can reformulate the statement~\ref{st4} of
Proposition~\ref{split-ext} in the following way:
splittings of the central extension
 $\widehat{GL_n(\da_{X}) }$ from statements~\ref{st1}-\ref{st3}  coincide over the intersections of corresponding subgroups.
 }
\end{nt}

\subsection{Second Chern number} \label{Chern-sect}

Now we give a construction of the second Chern number for a vector bundle on a smooth algebraic surface $X$ over $k$.

 Let $\mathcal E$ be a locally free sheaf of $\oo_X$-modules of rank $n$ on $X$.
Follow~\cite{Par} we introduce transition matrices for $\mathcal E$. For any point $x \in X$ let $e_x$ be a basis of the free $\hat{\oo}_{x,X}$-module
 ${\mathcal E} \otimes _{\oo_X} \hat{\oo}_{x,X}$. For any irreducible curve $C$ on $X$ let $e_C$ be a basis of the free $\oo_{K_C}$-module ${\mathcal E} \otimes_{\oo_X} \oo_{K_C} $, where $\oo_{K_C}$ is the discrete valuation ring of the field $K_C$.   Let $e_0$ be a a basis of the free $k(X)$-module ${\mathcal E} \otimes_{\oo_X}
 k(X) $. These expressions can be considered as completions of the stalks of $\mathcal E$ at scheme points of $X$.  Each of above basis consists of $n$ elements.

 For any point $x \in X$ we have the transition matrix $\alpha_x \in GL_n(K_x)$ defined as $ e_0= \alpha_x e_x$.
 For any irreducible curve $C$ on $X$ we have the transition matrix $\alpha_C \in GL_n(K_C)$ defined as $e_0= \alpha_C e_C$.
 For any pair $x \in C$ we have the transition matrix $\alpha_{x,C} \in GL_n(\oo_{K_{x,C}})$ defined as $e_x = \alpha_{x,C} e_C$.

 When we vary points $x \in X$, we obtain the  matrix $\alpha_{02} = \{ \alpha_x  \} \in GL_n(\prod_{x \in C} K_{x,C})$ via the diagonal embedding.
When we vary irreducible curves $C$ on $X$, we obtain the matrix $\alpha_{01} = \{ \alpha_C \}  \in GL_n(\prod_{x \in C} K_{x,C})$ via the diagonal embedding.
When we vary pairs $x \in C$, we  obtain the matrix $\alpha_{21}= \{ \alpha_{x,C} \}   \in GL_n(\prod_{x \in C} K_{x,C}) $. We define $\alpha_{20}
= \alpha_{02}^{-1}$, $\alpha_{10}= \alpha_{01}^{-1}$, $\alpha_{12}= \alpha_{21}^{-1}$. We have an evident equality:
\begin{equation}   \label{cocycle}
\alpha_{02}  \alpha_{21}  \alpha_{10} = 1 \mbox{,}
\end{equation}
where $1$ is the identity matrix.

If we change the basis:
$$
\{ e_x  \}  \longmapsto \alpha_2 \{e_x\}  \; \mbox{,} \qquad \{ e_C \}  \longmapsto \alpha_1  \{ e_C \}    \; \mbox{,}  \qquad e_0 \longmapsto
\alpha_0  e_0   \; \mbox{,}
$$
where $\alpha_2 \in GL_n(\prod_{x \in X} \hat{\oo}_{x,X}) = GL_n(\da_{X,02} \cap \da_{X, 12}) $,
$\alpha_1 \in GL_n(\prod_{C \subset X} \oo_{K_{C}} ) = GL_n(\da_{X, 01}  \cap \da_{X, 12})$, and $\alpha_0 \in GL_n (k(X))$, then we obtain the change of matrices:
\begin{equation}  \label{change}
\alpha_{02} \longmapsto \alpha_0 \alpha_{02} \alpha_2^{-1} \; \mbox{,} \qquad
\alpha_{21} \longmapsto \alpha_2 \alpha_{21} \alpha_1^{-1}  \; \mbox{,}  \qquad
\alpha_{10}   \longmapsto \alpha_1 \alpha_{10}  \alpha_0^{-1}  \mbox{.}
 \end{equation}

It is easy to see that $\alpha_{01}$ and $\alpha_{02}$ are from $GL_n (\da_X)$, because by formula~\eqref{change} we can change the basis $e_0$, $\{ e_x \}$ and $\{ e_C \}$ to a more convenient basis, for example, to take a trivialization of $E$ on some open cover of $X$ in Zariski topology, and then $e_0$ equal to the trivialization of $\mathcal E$ on a fixed open subset from this open cover, and $e_x$, $e_C$ also come from the trivialization of $\mathcal E$ on this open cover of $X$. Hence and using formula~\eqref{cocycle} we obtain that $\alpha_{21} \in GL_n(\da_{X})$. Therefore we have that
$$
\alpha_{02} \in GL_n(\da_{X, 02}) \; \mbox{,} \qquad
\alpha_{21}  \in GL_n(\da_{X, 12}) \; \mbox{,}  \qquad
\alpha_{10}   \in GL_n(\da_{X, 01})  \mbox{. }
$$

Let $\widehat{\alpha_{02}} \in \widehat{GL_n(\da_{X}) }$ be the canonical section applied to $\alpha_{02}$ and  which was constructed in statement~\ref{st2}  of Proposition~\ref{split-ext}.
Let $\widehat{\alpha_{21}} \in \widehat{GL_n(\da_{X}) }$ be the canonical section applied to $\alpha_{21}$ and  which was constructed in statement~\ref{st1}  of Proposition~\ref{split-ext}.
Let $\widehat{\alpha_{02}} \in \widehat{GL_n(\da_{X}) }$ be the canonical section applied to $\alpha_{02}$ and  which was constructed in statement~\ref{st3}  of Proposition~\ref{split-ext}.

\begin{Th}
\label{Chern}
 Let $\mathcal E$ be a locally free  sheaf of $\oo_X$-modules of rank $n$ on a smooth projective surface $X$ over a perfect field $k$.
\begin{enumerate}
\item An expression $\widehat{\alpha_{02}}  \, \widehat{\alpha_{21}} \, \widehat{\alpha_{10}} \in \widehat{GL_n(\da_{X}) }$ gives an element from $\Z$ and does not depend on the choose of  basis $e_0$, $\{ e_x \}$ and $\{ e_C \}$ of $\mathcal E$.
\item An expression $\widehat{\alpha_{02}} \, \widehat{\alpha_{21}} \, \widehat{\alpha_{10}} $ equals to the second Chern number $c_2({\mathcal E})$
of $\mathcal E$.
\end{enumerate}
\end{Th}
\begin{proof}

1. Since the image of  $\widehat{\alpha_{02}}  \, \widehat{\alpha_{21}} \, \widehat{\alpha_{10}} $ in $GL_n(\da_X)$
equals to $\alpha_{02}  \alpha_{21}  \alpha_{10} = 1$, we obtain that   $\widehat{\alpha_{02}}  \, \widehat{\alpha_{21}} \, \widehat{\alpha_{10}} \in \Z$. The independence on the choice of basis follows from formula~\eqref{change} and the statement~\ref{st4} of Proposition~\ref{split-ext}.

2. It is known that for any locally free sheaf $\mathcal F$ of $\oo_X$-modules of rank more than $1$ there is a smooth surface $Y$ and a morphism $f : Y \to X$,
where $Y$ is obtained by means of chain of of blow-ups of points, such that there is a locally free subsheaf ${\mathcal F_1} \subset f^* {\mathcal F}$ with $(f^* {\mathcal F}) / {\mathcal F_1}$
is again a locally free subsheaf of $\oo_Y$-modules. (Indeed, it is enough to find a section $s \in H^0(Y, f^* {\mathcal E})$ such that $s(y)\ne 0$ for any point $y \in Y$, then   the sheaf  ${\mathcal M} = (f^* {\mathcal E}) / (\oo_{Y} \cdot s)$ is locally free, because $\mathop{\rm Tor}_1^{\oo_{y,Y}}(k(y), {\mathcal M}_y)=0$ for any point $y \in Y$.)

Therefore from the general theory of the Chern classes it follows that the number $\tilde{c}_2({\mathcal E})$ coincides with the second Chern number $c_2({\mathcal E})$
on a smooth projective algebraic surface
if and only if the following conditions are satisfied:

1) $\tilde{c}_2({\mathcal L})=0$ for any locally free sheaf $ \mathcal L$ of rank $1$;

2) $\tilde{c}_2({\mathcal N})= \tilde{c}_2(\pi^*( {\mathcal N}))$, where $\mathcal N$ is a locally free sheaf and $\pi$ is a blow-up of a point;

3) for any exact sequence of locally free sheaves
\begin{equation} \label{ex-triple}
0 \lrto {\mathcal E}_1 \lrto {\mathcal E}_2 \lrto {\mathcal E}_3 \lrto 0
\end{equation}
we have $\tilde{c}_2({\mathcal E}_2) = \tilde{c}_2({\mathcal E}_1)   + \tilde{c}_2({\mathcal E}_3)  + (\det({\mathcal E}_1),
\det({\mathcal E}_3))$, where $( \cdot ,  \cdot)$ is the intersection index of two divisors which are rational sections of corresponding invertible sheaves.

In our case $\tilde{c}_2$ equals to $\widehat{\alpha_{02}} \, \widehat{\alpha_{21}} \, \widehat{\alpha_{10}} $. The first condition is satisfied, because, by construction, $\widehat{GL_1(\da_{X}) }= \Z \times \da_{X}^*$.

To check the second condition we note that if $\pi : Y \to X$ is a blow-up of a point $x \in X$, then $\da_{Y}= \da_X \times \da_{\Delta}$ with the set $\Delta$ which consists of all pairs $y \in R$, where $\pi(R)=x$.
By the first statement of this theorem, $\tilde{c}_2$ does not depend on the choice of the basis. Therefore we choose the special basis. We fix a trivialization of $\mathcal E$ on an open neighbourhood of $x$ on $X$. This trivialization gives us the same basis $e_0$, $e_x$, $e_R$ and $e_y$, where $y \in R$, for $\mathcal E$ and $\pi^* {\mathcal E}$.
  We identify the other basis for $\mathcal E$ and $\pi^* {\mathcal E}$.
The decomposition ${GL_n(\da_{Y})= GL_n(\da_X) \times GL_n(\da_{\Delta})}$ implies the canonical embedding
${
\gamma : GL_n(\da_X)  \hookrightarrow GL_n(\da_Y)
}
$.
From construction of the central extension we have canonical isomorphism  (compare also with Remark~\ref{dec-hat}):
$$ \delta: \quad  \widehat{GL_n(\da_X)}  \lrto \gamma^* (\widehat{GL_n(\da_Y)}  \mbox{.}
$$
It is easy to see that from our choice of the basis for $\mathcal E$ and $\pi^* {\mathcal E}$  we have
$$\gamma(\alpha_{02, {\mathcal E}})= \alpha_{02, \pi^* {\mathcal E}}  \; \mbox{,} \qquad
\gamma(\alpha_{21, {\mathcal E}})= \alpha_{21, \pi^* {\mathcal E}} \; \mbox{,} \qquad
\gamma(\alpha_{10, {\mathcal E}})= \alpha_{10, \pi^*  {\mathcal E}} \mbox{,}$$
where we put an additional index $\mathcal E$ or $\pi^* {\mathcal E}$ to specify a locally free sheaf.
 Besides, from the construction of the central extension and the splittings we obtain
 $$
 \delta(\widehat{\alpha_{02, {\mathcal E}} })= \widehat{\alpha_{02, \pi^* {\mathcal E}}} \; \mbox{,} \qquad
\delta(\widehat{\alpha_{21, {\mathcal E}} })= \widehat{\alpha_{21, \pi^* {\mathcal E})}} \;  \mbox{,}  \qquad
\delta(\widehat{\alpha_{10, {\mathcal E}} })= \widehat{\alpha_{10, \pi^* {\mathcal E}}} \mbox{,}$$
where we consider elements $\widehat{\alpha_{02, \pi^* {\mathcal E}}}$, $ \widehat{\alpha_{21, \pi^*  {\mathcal E}}}$ and  $\widehat{\alpha_{10, \pi^* {\mathcal E}}} $
as elements from the group $\gamma^* (\widehat{GL_n(\da_Y)} $.
This finishes the checking of the second condition.

Since we can take the basis compatible with exact sequence~\eqref{ex-triple} and $\tilde{c}_2$ does not depend on the choice of the basis, the third condition for $\tilde{c}_2$ follows from statements~\ref{st5}  and~\ref{st6} of Proposition~\ref{split-ext}, the second statement of Proposition~\ref{ext-restr} and the following fact.  Let $\mC$ and $\mD$ be invertible sheaves on a smooth projective surface $X$, and $\alpha_{02, {\mathcal \mC}} $,  $\alpha_{21, {\mathcal \mC}} $,  $\alpha_{10, {\mathcal \mC}} $,
 $\alpha_{02, {\mathcal \mD}} $,  $\alpha_{21, {\mathcal \mD}} $,  $\alpha_{10, {\mathcal \mD}} $ be transition matrices, in fact elements from $\da_{X}^*$, for sheaves $\mC$ and $\mD$
 correspondingly (after the choice of basis for these sheaves).  Then there is an equality in the group
 $\widehat{\da_X^* \times \da_X^*}$:
 \begin{equation}  \label{equal}
 (\alpha_{02, \mC}, \alpha_{02,\mD}; 0) (\alpha_{21, \mC}, \alpha_{21, \mD}; 0) (\alpha_{10, \mC}, \alpha_{10, \mD}; 0)= (\mC,  \mD )  \; \in \; \Z  \mbox{.}
 \end{equation}
 We prove this equality now. The left hand side of expression~\eqref{equal} equals to
 \begin{multline*}
 (\alpha_{02, \mC } \cdot \alpha_{21, \mC}, \, \alpha_{02, \mD}  \cdot \alpha_{21, \mD}; \, \langle \alpha_{21, \mC}, \alpha_{02, \mD}   \rangle_X)
 \cdot
 (\alpha_{10, \mC}, \alpha_{10, \mD}; 0) = \\ = (\alpha_{01, \mC},\alpha_{01, \mD},  \langle \alpha_{21, \mC}, \alpha_{02, \mD}   \rangle_X) \cdot
(\alpha_{10, \mC}, \alpha_{10, \mD}; 0) =
(1,1;  \langle \alpha_{21, \mC}, \alpha_{02, \mD}   \rangle_X) + \langle \alpha_{10, \mC},  \alpha_{01, \mD}  \rangle_X)
 \end{multline*}
Since by statements~\ref{st1}-\ref{st3} of Proposition~\ref{split-ext} the central extension $\widetilde{GL_1(\da_X)}$ canonically splits over subgroups $\da_{X, 12}^*$, $\da_{X, 01}^*$ and $\da_{X, 02}^*$, we have $\langle \alpha_{10, \mC},  \alpha_{01, \mD}  \rangle_X =0$. Therefore it is enough to prove that $\langle \alpha_{21, \mC}, \alpha_{02, \mD}   \rangle_X= (\mC, \mD)$.  We have
\begin{multline*}
\langle \alpha_{21, \mC}, \alpha_{02, \mD}   \rangle_X = \langle \alpha_{02, \mD},  \alpha_{12, \mC}  \rangle_X =
 \langle \alpha_{02, \mD},  \alpha_{12, \mC}  \rangle_X + \langle  \alpha_{21, \mD}, \alpha_{12, \mC}      \rangle_X =
 \langle \alpha_{01, \mD}, \alpha_{12, \mC}      \rangle_X  =\\
 =  \langle \alpha_{01, \mD}, \alpha_{01, \mC}      \rangle_X    + \langle \alpha_{01, \mD}, \alpha_{12, \mC}      \rangle_X =
  \langle \alpha_{01, \mD}, \alpha_{02, \mC} \rangle_X = (\mC, \mD)  \mbox{,}
 \end{multline*}
 where the last equality follows from Proposition~\ref{intersect}.

 Thus we have proved the theorem.
\end{proof}

\section{Case of arithmetic surface}   \label{arsur}

Now we will give analogs of certain statements of Proposition~\ref{split-ext} for arithmetic surfaces. We note that Proposition~\ref{split-ext}
is one of key propositions used in Theorem~\ref{Chern}.

By an arithmetic surface we mean here
a two-dimensional integral regular scheme of finite type over $\Z$ with the proper surjective morphism to $\Spec \Z$. For an arithmetic surface $X$
there is an adelic arithmetic ring $\da_{X}^{\rm ar}$ introduced in~\cite[Example~11]{OsipPar2} (see also explanations in~\cite[\S~3.4]{Osi2}):
$$
\da_{X}^{\rm ar }  \eqdef \da_X  \times \da_{X, \infty}  \; \mbox{,}
$$
where the ring
$$\da_{X, \infty}  \eqdef \da_{X_{\Q}} \, \widehat{\otimes} \, \R \;
=
 \mathop{\Lim_{\lrto}}_{D_2} \mathop{\Lim_{\longleftarrow}}_{D_1 \ge D_2}  (\da_{X_{\Q}}(D_2)/ \da_{X_{\Q}}(D_1)) \otimes_{\Q}   \R
 \mbox{,}$$
and $\da_{X_{\Q}}$ is the adelic ring of the curve  ${X_{\Q} = X \times_{\Spec \Z} \Spec \Q }$, which is the generic  fibre,
$D_1$ and $D_2$ are divisors on the curve $X_{\Q}$.

The central extensions $\widetilde{GL_n(\da^{\rm ar}_X)}_{\dr_+^*}$ and $\widehat{GL_n(\da^{\rm ar}_{X})}_{\dr_+^*}$ of the group
$GL_n(\da_X^{\rm ar})$ by the (multiplicative) group of positive real numbers $\dr_+^*$ were constructed in~\cite{Osi2}.
The constructions of these central extensions can be done similarly to constructions of central extensions  $\widetilde{GL_n(\da_{\Delta})}$ and $\widehat{GL_n(\da_{\Delta}) }$ from sections~\ref{centr-ext1} and~\ref{centr-ext2}, but we have to use that $\da_X^{\rm ar}$ is an object of the category $C_2^{\rm ar}$ from~\cite[\S~5]{OsipPar2} (instead of the category $C_2$ which we used before). Correspondingly, instead of locally linearly compact $k$-vector spaces we have to use locally compact Abelian groups, and instead of $\Z$-torsor of dimension theories $\Dim(V)$ for a locally linearly compact $k$-vector space $V$ we have to use $\dr_+^*$-torsor $\mu(W)$ of Haar measures for a locally compact Abelian group $W$.

We note that the similar construction can be done  for an algebraic surface over a finite field $\df_q$. In this case a locally linearly compact $\df_q$-vector space $V$ is also an Abelian locally compact group. A homomorphism $\Z  \to \dr_+^* : a \mapsto q^{a}$ induces the map $\Dim(V) \to \mu(V)$
of corresponding torsors. This gives the homomorphism from the central extensions  $\widetilde{GL_n(\da_{X})}$ and $\widehat{GL_n(\da_{X}) }$ to the central extensions  $\widetilde{GL_n(\da^{\rm ar}_X)}_{\dr_+^*}$ and $\widehat{GL_n(\da^{\rm ar}_{X})}_{\dr_+^*}$ correspondingly.

Similarly to Lemma~\ref{lem:Baer}  we have that the cental extension  $\widetilde{GL_n(\da^{\rm ar}_X)}_{\dr_+^*}$  is the Baer sum of the central extensions  $\widetilde{GL_n(\da_X)}_{\dr_+^*}$ and
$\widetilde{GL_n(\da_{X, \infty})}_{\dr_+^*}$, where the last two central extensions are obtained by restrictions
of the central extension $\widetilde{GL_n(\da^{\rm ar}_X)}_{\dr_+^*}$
to subgroups $GL_n(\da_X)$ and $GL_n(\da_{X, \infty})$ of the group $GL_n(\da_X^{\rm ar})$. Analogously to this statement and similarly to  Remark~\ref{dec-hat},
the central extension $\widehat{GL_n(\da^{\rm ar}_{X})}_{\dr_+^*}$  is the Baer sum of the central extensions  $\widehat{GL_n(\da_X)}_{\dr_+^*}$ and
$\widehat{GL_n(\da_{X, \infty})}_{\dr_+^*}$.

We have a subring $\da_{X, \infty}(0) = \mathop{\Lim\limits_{\longleftarrow}}\limits_{D \le 0}  (\da_{X_{\Q}}(0)/ \da_{X_{\Q}}(D)) \otimes_{\Q}   \R$  of the ring $\da_{X, \infty}$, where $0$ is zero divisor on $X_{\Q}$, and $D $ is a divisor on $X_{\Q}$ which is less or equal than $0$.

The ring $\da_{X, 02}$ (with the definition as in formula~\eqref{subrings}) is a subring of the ring  $\da_{X}$. Besides, the restrictions of the central extensions  $\widetilde{GL_n(\da^{\rm ar}_{X})}_{\dr_+^*}$ and $\widehat{GL_n(\da^{\rm ar}_{X})}_{\dr_+^*}$ to the subgroup $GL_n(\da_{X, 02})$ embedded to the group $GL_n(\da_X^{ar})$ as $g \mapsto g \times 1$
coincide (or canonically isomorphic) with the restrictions of the  central extensions    $\widetilde{GL_n(\da_{X})}_{\dr_+^*}$ and $\widehat{GL_n(\da_{X})}_{\dr_+^*}$ to the subgroup $GL_n(\da_{X, 02})$ of the group $GL_n(\da_X)$.

The ring $\da_{X, 01}$ (with the definition as in formula~\eqref{subrings}) is a subring of the ring $\da_X$. Besides, there is a homomorphism from the ring $\da_{X, 01}$ to the ring $\da_{X}^{\rm ar}$ induced by the natural embedding $K_C \hookrightarrow \da_{X, \infty}$ for any ``horizontal''
curve $C$ on $X$ (or, in other words, $C$ is an integral one dimensional subscheme of $X$ which maps surjectively onto $\Spec \Z$), see more explanations in~\cite[\S~3.4-\S~3.5]{Osi2}. Thus we consider $\da_{X, 01}$ as a subring of $\da_{X}^{\rm ar}$, where $\da_{X, 01}$ is mapped in the both parts of $\da_{X}^{\rm ar}$. The last embedding induces the embedding
$GL_n(\da_{X, 01}) \hookrightarrow GL_n(\da_X^{\rm ar})$.

We obtain the following proposition, which contains analogs of statements~\ref{st1}-\ref{st3}  from Proposition~\ref{split-ext}
and generalizes Theorem~1 from~\cite{Osi2}.
\begin{prop} \label{arithm}
Let $X$ be an arithmetic surface.
 The central extensions  $\widetilde{GL_n(\da^{\rm ar}_{X})}_{\dr_+^*}$ and $\widehat{GL_n(\da^{\rm ar}_{X})}_{\dr_+^*}$ canonically split over the subgroups $GL_n(\da_{X, 12}) \times GL_n(\da_{X, \infty}(0))$, $GL_n(\da_{X, 02})$ and $GL_n(\da_{X, 01})$ of the group $GL_n(\da_{X}^{\rm ar})$.
\end{prop}
The proof of this proposition is completely similar to the proof of analogous statements of Proposition~\ref{split-ext}. We note only that  the splitting over the subgroup $GL_n(\da_{X, 02})$ is enough to prove for the central extensions $\widetilde{GL_n(\da_{X})}_{\dr_+^*}$ and $\widehat{GL_n(\da_{X})}_{\dr_+^*}$.

\vspace{0.3cm}

\noindent Steklov Mathematical Institute of Russsian Academy of Sciences \\
\noindent  ul. Gubkina 8, Moscow, 119991 Russia \\

\medskip
\noindent {\it E-mail:}  ${d}_{-} osipov@mi.ras.ru$

\end{document}